\documentclass[reqno]{amsart}
\usepackage[a4paper,left=2cm,right=2cm,top=2cm,bottom=2cm]{geometry}
\usepackage[foot]{amsaddr}

\usepackage{booktabs}
\usepackage[svgnames,table]{xcolor}
\usepackage[tableposition=above]{caption}
\usepackage{pifont}

\usepackage{enumitem}
\usepackage{amsmath}
\usepackage{amsthm}
\usepackage{amsfonts}
\usepackage{amssymb}
\usepackage{dsfont}
\allowdisplaybreaks

\usepackage{tikz}
\usepackage{tensor}

\usepackage[
		pdftex,
		bookmarks,
		bookmarksopen,
		bookmarksopenlevel=1,
		bookmarksnumbered,
		breaklinks,
		colorlinks,
		linkcolor=linkcolor,
		citecolor=citecolor,
		urlcolor=urlcolor,
	]{hyperref}
\hypersetup{
	pdftitle    = {},
	pdfauthor   = {Andreas Vollmer}
}
\definecolor{linkcolor}{rgb}{0.5,0.0,0.0}
\definecolor{citecolor}{rgb}{0.0,0.5,0.0}
\definecolor{urlcolor} {rgb}{0.0,0.0,0.5}

\usepackage{pgfplots}
\usetikzlibrary{calc}
\pgfplotsset{compat=1.8}

\usepackage[backend=bibtex,giveninits=true,url=false,isbn=false,doi=false,style=numeric,maxnames=5]{biblatex}
\theoremstyle{plain}

\newtheorem{lemma}{Lemma}
\newtheorem{theorem}{Theorem}
\newtheorem{proposition}{Proposition}
\newtheorem{corollary}{Corollary}

\theoremstyle{definition}
\newtheorem{definition}{Definition}
\newtheorem{example}{Example}

\theoremstyle{remark}
\newtheorem{remark}{Remark}

\newcommand{\projcl}{\ensuremath{\mathfrak{P}}}
\newcommand{\projstr}{\ensuremath{\mathcal{P}}}
\newcommand{\solSp}{\ensuremath{\mathcal{M}}}

\newcommand{\lie}{\ensuremath{\mathcal{L}}}

\begin{document}

\title[Projectively equivalent superintegrable systems]{Projectively equivalent 2-dimensional superintegrable systems\\ with projective symmetries}
\author{Andreas Vollmer}
\email{andreasdvollmer@gmail.com}
\email{a.vollmer@unsw.edu.au}
\address{School of Mathematics and Statistics, University of New South Wales, Sydney NSW 2052, Australia}

\begin{abstract}
This paper combines two classical theories, namely metric projective differential geometry and superintegrability.
We study superintegrable systems on 2-dimensional geometries that share the same geodesics, viewed as unparametrized curves.
We give a definition of projective equivalence of such systems, which may be considered the projective analog of (conformal) St\"ackel equivalence (coupling constant metamorphosis).
Then, we discuss the transformation behavior for projectively equivalent superintegrable systems and find that the potential on a projectively equivalent geometry can be reconstructed from a characteristic vector field.
Moreover, potentials of projectively equivalent Hamiltonians follow a linear superimposition rule.
The techniques are applied to several examples. In particular, we use them to classify, up to St\"ackel equivalence, the superintegrable systems on geometries with one, non-trivial projective symmetry.
\end{abstract}

\maketitle

\begin{center}
 \scriptsize
 \vskip-2\baselineskip
 keywords: geodesic equivalence, projective connections, superintegrable systems, projective symmetry\smallskip
 
\end{center}
\vskip\baselineskip

\section{Introduction}\label{sec:introduction}
Let $(M,g)$ be a (pseudo-)Riemannian manifold of dimension~2. A (parametrized) geodesic $\gamma=\gamma(t)$ is a curve on $M$ that satisfies the equation
\begin{equation}\label{eqn:geodesic}
  \nabla_{\dot\gamma}\dot\gamma = 0
\end{equation}
where $\nabla$ is the Levi-Civita connection of $g$.
Note that~\eqref{eqn:geodesic} requires the geodesic to be parametrized in a certain way. If we release this requirement, and reparametrize using a new parameter $s=s(t)$, Equation~\eqref{eqn:geodesic} becomes
\begin{equation*}
 \nabla_{\gamma'}\gamma' = \chi\,\gamma'
\end{equation*}
where we denote derivatives w.r.t.~$s$ by $'$, and where the coefficient function is given by $\chi=g(\gamma',\nabla\ln\dot s)$.
We investigate geometries that share the same geodesics up to their parametrization.
\begin{definition}
Two pseudo-Riemannian metrics are \emph{projectively} (aka \emph{geodesically}) \emph{equivalent} if their Levi-Civita connections give rise to the same geodesics, viewed as unparametrized curves.
\end{definition}

\noindent This paper is devoted to superintegrable systems on geometries that are projectively equivalent in this sense (a proper definition is going to be given in Section~\ref{sec:equivalent.systems}).
A superintegrable system admits a maximal number of independent integrals of motion, i.e.\ of (smooth) functions $f:TM\to\mathbb{R}$ that are preserved along geodesics, $f(\gamma(t),\dot\gamma(t))=\text{const}$.
Integrals of motion, and their properties under geodesic transformations, have been the subject of a number of classical works in mathematics, going back at least to the middle of the 19th century, e.g.~\cite{joachimsthal_1846,raabe_1846,painleve_1894,levi-civita_1896,liouville_1889,darboux_1901} to name just a few.
Interest in these objects has increased in recent years, with advances, for instance, in the area of metrizability \cite{matveev_2012rel,eastwood_2008,bryant_2008}, an invariant description of Killing tensors \cite{gover_2018}, invariants \cite{fels_1995}, and many more.

\subsection{Superintegrable systems}
In the current section, superintegrable systems are going to be introduced.
We are going to restrict to what is known as \emph{second order maximally superintegrable} systems, i.e.\ we require the integrals to be quadratic polynomials in the momenta.
More precisely, let $H:T^*\!M\to\mathbb{R}$, $H=g^{ij}p_ip_j+V$, be the natural Hamiltonian of the metric $g$ endowed with the \emph{potential} $V:M\to\mathbb{R}$.
We denote momenta on the cotangent space $T^*\!M$ by $p$, and velocities on the tangent space $TM$ by $\xi$.
For the free Hamiltonian we write $G:T^*\!M\to\mathds{R}$, $G=g(\xi,\xi)=g_{ij}\xi^i\xi^j=g^{ij}p_ip_j$, which is a polynomial of second order in the momenta or velocities, respectively (note that here, as in the remainder of this paper, the Einstein summation convention is used).
Integrals of motion $I:T^*\!M\to\mathbb{R}$ are characterized, in Hamiltonian mechanics, by the vanishing of their Poisson bracket with the Hamiltonian, i.e.
\begin{equation}\label{eqn:poisson}
 \{H,I\} = \frac{\partial H}{\partial x^i}\frac{\partial I}{\partial p_i}
        -\frac{\partial H}{\partial p_i}\frac{\partial I}{\partial x^i}
 = 0\,.
\end{equation}
It is a well-known fact that, provided the Hamiltonian has the above form, that an integral which is a quadratic polynomial in momenta, may without loss of generality be taken to have the form~\eqref{eqn:integrals}, see e.g.~\cite{hietarinta_1984}.
\begin{definition}
A superintegrable system is a pseudo-Riemannian manifold of dimension $n$ together with $2n-1$ functionally independent integrals of motion, whereof one is the Hamiltonian $H=G+V$, $G=g^{ij}p_ip_j$. The other $2n-2$ integrals of motion have the form
\begin{equation}\label{eqn:integrals}
 I^{(\alpha)} = K^{ij}_{(\alpha)}p_ip_j + W^{(\alpha)}\qquad \alpha\in\{1,2,\dots,2n-2\}\,,
\end{equation}
where $K^{ij}_{(\alpha)}$ are components of a (2,0)-tensor field $K_{(\alpha)}$, and $W^{(\alpha)}\in C^\infty(M)$ is a function on the underlying manifold.
\end{definition}

\noindent The condition~\eqref{eqn:poisson}, for each $I^{(\alpha)}$, is a polynomial in the momenta. If we decompose them, for each value of $\alpha$, according to the (polynomial) degree in the momenta, we obtain a system of equations
\begin{equation}\label{eqn:poisson.split}
 \{G,K^{ij}_{(\alpha)}p_ip_j\} = 0\,,\qquad
 \{G,V\}+\{W^{(\alpha)},K^{ij}_{(\alpha)}p_ip_j\} = 0\,,\qquad
 \alpha\in\{1,2,\dots,2n-2\}\,.
\end{equation}
Let us denote by $K_{ij}^{(\alpha)}=g_{ia}g_{jb}K^{ab}_{(\alpha)}$ the components of the (0,2)-tensor field $K^{(\alpha)}$ corresponding to $K_{(\alpha)}$. Then the first of the conditions~\eqref{eqn:poisson.split} corresponds to the requirement that the tensor field \smash{$K^{(\alpha)}=K_{ij}^{(\alpha)}dx^idx^j$} is a Killing tensor, i.e.
\begin{equation}\label{eqn:killing}
 \nabla_XK^{(\alpha)}(X,X)=0
\end{equation}
for any tangent vector field $X$.
Similarly, we may introduce an endomorphism
\begin{equation*}
  \mathsf{K}^{(\alpha)}:T^*\!M\to T^*\!M\,,
  \qquad
  \mathsf{K}^{(\alpha)}(\omega)=K\indices{^{(\alpha)}_i^j}\omega_j\,dx^i\,.
\end{equation*}
and in these terms the second requirement of~\eqref{eqn:poisson.split} gives an expression for the differential $dW^{(\alpha)}$. Its integrability condition is
\begin{equation}\label{eqn:bertrand.darboux}
  d\mathsf{K}^{(\alpha)}dV = 0\,,
\end{equation}
i.e.\ $d(K^{(\alpha)}_{ia}\nabla^aV)=0$.
Equation~\eqref{eqn:bertrand.darboux} is known as \emph{Bertrand-Darboux equation} \cite{bertrand_1857,darboux_1901}.
If $V\ne0$, we require the full system~\eqref{eqn:poisson.split}, and particularly~\eqref{eqn:bertrand.darboux}, to hold. If $V=0$, we deal with a free Hamiltonian system and only need to account for~\eqref{eqn:killing}.

\begin{remark}[Non-degeneracy]\label{rmk:non.degeneracy}
 A superintegrable metric may admit an entire linear family of potentials that are compatible, via~\eqref{eqn:bertrand.darboux}, with the same space of Killing tensor fields. Particularly, if a 2-dimensional metric admits a superintegrable Hamiltonian with the potential being a $n+2=4$ dimensional vector space of potentials, the superintegrable system is said to be \emph{non-degenerate}. Otherwise, it is called \emph{degenerate}.
\end{remark}

Two famous superintegrable systems are the Kepler-Coulomb system and the Harmonic Oscillator, which both have major significance in many areas of science, ranging from atomic physics and materials science to celestial mechanics and quantum theory.
For instance, in classical celestial mechanics, the Kepler 2-body problem (planetary motion around a central body) is superintegrable and solvable by quadrature due to the existence of the Runge-Lenz vector. Due to angular momentum conservation it can be reduced to a 2-dimensional problem. In quantum mechanics, the corresponding problem is the determination of the energy level structure of the Hydrogen atom, which also emphasizes the close link between superintegrability and separation of variables~\cite{kalnins_2018}. Examples also include, for instance, oscillations in crystalls and metals and the Calogero-Moser model.

Second-order maximally superintegrable systems have been classified in dimension 2 and 3, see \cite{kalnins_2018} and references therein.
Particularly, in dimension~2, superintegrable systems are classified in terms of normal forms, see \cite{kalnins_2001,kalnins-II,kalnins_2018}, and at least for the Euclidean case, the algebraic geometry underlying non-degenerate systems is understood \cite{kalnins_2007,kress_2019}. For the 3-dimensional case, see~\cite{capel_2014} for instance.
%
Second order superintegrable systems admit a notable equivalence transformation, known as St\"ackel transformations or coupling constant metamorphosis. Both names are interchangeable for our context, but in general the transformations are not identical~\cite{post_2010}).
St\"ackel transformations have a major significance in superintegrability, providing an equivalence relation on second-order superintegrable systems \cite{kalnins-II,kalnins_2011,kress_2007}.
They also play an important role in the classification of superintegrable systems.
The relevant theory is concisely summarized in the literature, e.g.\ \cite{kress_2007,kalnins_2018}, and we therefore mention only the key points relevant to our purposes in this paper.
\begin{definition}[St\"ackel transformation]
 Consider a Hamiltonian $H=H_0+\epsilon V_0$ with coupling parameter $\varepsilon$, admitting an integral of motion $L=L_0+\epsilon W_0$ and satisfying
 \[
  \{H_0,L_0\}=0=\{H,L\}\,.
 \]
 Then, the St\"ackel transformed objects
 \[
  H'=\frac{H_0}{V_0}
  \qquad\text{and}\qquad
  L'=L_0-W_0H'
 \]
 satisfy $\{H',L'\}=0$.
\end{definition}
\noindent The St\"ackel transform establishes an equivalence relation on second-order maximally superintegrable systems, see e.g.\ \cite{kalnins-II}.
In fact, it is often beneficial to consider superintegrable systems up to St\"ackel equivalence, for instance in~\cite{kalnins_2013} and particularly in the classification of superintegrable systems, e.g.\ \cite{capel_2014,capel_2015}.
Let us mention some properties of St\"ackel transformations that are going to be important in what follows:
\begin{enumerate}
 \item A superintegrable system in dimension 2 is given by the Hamiltonian $H$ and two integrals $I^{(1)},I^{(2)}$.
 Denote their (non-vanishing) Poisson bracket by $R=\{I^{(1)},I^{(2)}\}$.
 It is proven in \cite{kalnins-I,kalnins-II} that $R^2$ is a cubic polynomial in $H,I^{(1)},I^{(2)}$.
 \item It is easily recognized that the cubic $R^2$ is not canonical, as we are free to replace $I^{(1)},I^{(2)}$ by linear combinations, including with the Hamiltonian and constant terms,
 \begin{align*}
 \hat{R}^2
 &= \{ a_1 I^{(1)} + b_1 I^{(2)} + c_1 H + d_1 , a_2 I^{(1)} + b_2 I^{(2)} + c_2 H + d_2 \}^2 \\
 &=(a_1b_2-a_2b_1)^2\,\{I^{(1)},I^{(2)}\}^2
 =(a_1b_2-a_2b_1)^2\,R^2\,.
 \end{align*}
 This ambiguity can be addressed by resorting to normal forms.
 Such normal forms are obtained in~\cite{kress_2007}, and we may put the cubic $R^2$ into one of the forms of Table~\ref{tab:staeckel}.
 \item Once in normal form, the St\"ackel type of a 2-dimensional superintegrable system (i.e.\ its equivalence class under St\"ackel equivalence) can be determined from the cubic $R^2$ \cite{kalnins-II,kress_2007}.
 The relevant terms of $R^2$ are the terms cubic and quadratic in the integrals $I^{(1)},I^{(2)}$. 
 This is based on the following observation: Let us split up the cubic $R^2$ according to the polynomial degree in $I^{(1)},I^{(2)}$. 
 Then, as shown in~\cite{kress_2007}, the leading term w.r.t.\ $I^{(1)}$ and $I^{(2)}$ is preserved under St\"ackel transformation, and for the part quadratic in $I^{(1)},I^{(2)}$ at least the form is preserved (the functions $f(H,c_i)$ can be changed by adding multiples of the Hamiltonian, or constants, to the integrals $I^{(1)},I^{(2)}$ etc.).
\end{enumerate}
Exploiting these three properties, we are going to determine the St\"ackel type of the superintegrable systems in Section~\ref{sec:main.example}.

\begin{table}
\begin{center}
 \begin{tabular}{crll}
  \toprule
  St\"ackel type & \multicolumn{3}{c}{Normal Form of $R^2$} \\
  \midrule
  \rowcolor{gray!10}
  (111,11) & $I^{(1)}I^{(2)}(I^{(1)}+I^{(2)})$ & $+f(H,c_i)\,I^{(1)}I^{(2)}$ & $+\mathcal{O}$ \\
  \midrule
  \rowcolor{gray!10}
  (21,2) & $(I^{(1)})^2I^{(2)}$ & $+f(H,c_i)\,(I^{(2)})^2$ & $+\mathcal{O}$ \\
  \rowcolor{gray!10}
  (21,0) & $(I^{(1)})^2I^{(2)}$ & & $+\mathcal{O}$ \\
  \midrule
  \rowcolor{gray!10}
  (3,11) & $(I^{(1)})^3$ & $+f(H,c_i)\,I^{(1)}I^{(2)}$ & $+\mathcal{O}$ \\
  (3,2) & $(I^{(1)})^3$ & $+f(H,c_i)\,(I^{(2)})^2$  & $+\mathcal{O}$ \\
  (3,0) & $(I^{(1)})^3$ & & $+\mathcal{O}$ \\ 
  \midrule
  (0,11) & & $f(H,c_i)\,I^{(1)}I^{(2)}$ & $+\mathcal{O}$ \\
  \bottomrule
 \end{tabular}
 \medskip
 
 \caption{The normal forms of St\"ackel types in dimension~2, as established in \cite{kress_2007}. The first column gives the St\"ackel type, while the second column shows the normal form into which $R^2$ can be cast. The symbol $\mathcal{O}$ indicates additional, lower degree terms; the coefficients $f(H,c_i)$ are linear polynomials in $H,c_i$ with constant coefficients. The constants $c_i$ denote the parameters of the non-degenerate potential (cf.\ Remark~\ref{rmk:non.degeneracy}). Shaded rows highlight classes that are realized in the example considered in Section~\ref{sec:main.example}.}\label{tab:staeckel}
\end{center}
\end{table}

\subsection{Structure of the paper}
The paper is organized as follows.
Section~\ref{sec:projective.equivalence} recalls the theory of projective equivalence and metrizability of projective connections. Next, in Section~\ref{sec:equivalent.systems}, we define what it means for two superintegrable systems to be projectively equivalent. Given a pair of projectively equivalent metrics $g,\tilde{g}$, we then explain how a superintegrable system admitted by the metric~$g$ can be transformed into one admitted by the projectively equivalent metric~$\tilde{g}$, and how these systems give rise to an ``addition'' of these superintegrable systems.

While the proofs turn out not to be too hard, we are going to find that the resulting techniques provide useful tools in the study of projectively equivalent systems. Examples are given in Section~\ref{sec:examples}.
The main application, however, is found in Section~\ref{sec:main.example}: We classify metrics with one, essential projective symmetry, up to St\"ackel equivalence. The paper is concluded with a remark on the interrelation of the appearing St\"ackel classes in Section~\ref{sec:contractions}.

\section{Brief review of projective differential geometry}\label{sec:projective.equivalence}
We begin with a short review of projective differential geometry, which has undergone some significant activity in recent years, see \cite{bryant_2009,mettler_2014a,mettler_2014b,gover_2018}, for instance. In particular, Lie's Problem of classifying 2-dimensional geometries with projective symmetries has been solved \cite{bryant_2008,matveev_2012,manno_2018}.
In dimension~2, there is also a close relationship with integrability~\cite{topalov_2003,bolsinov_2003,gover_2018}, which we are going to come back to in Section~\ref{sec:equivalent.systems}.

But for now let us begin with the following natural question: To what extend is it possible to reconstruct a geometry from the knowledge of its (unparametrized) geodesics?\nocite{matveev_2012rel}
Let $\nabla$ be the Levi-Civita connection of $g$. A projectively equivalent connection $\nabla'\sim\nabla$ satisfies
\begin{equation}
 \nabla'_aX^j = \nabla_aX^j + \Upsilon_aX^j + \Upsilon_cX^c\,g_a^j\,,
\end{equation}
for some 1-form $\Upsilon$,
and admits the same geodesic curves (disregarding their parametrization).
We denote the \emph{projective structure}, i.e.\ the collection of all connections projectively equivalent to $\nabla$ by $\projstr=[\nabla]$.
The projective class of a connection is encoded in its \emph{Thomas symbols}, which are given from the Christoffel symbols $\Gamma\indices{^k_{ij}}$ of $\nabla$ by the formula~\cite{thomas_1925,thomas_1926}
\begin{equation*}
 \Pi\indices{^k_{ij}} = \Gamma\indices{^k_{ij}} - \frac{1}{n+1}\,\delta^i_j\,\Gamma\indices{^p_{pk}} - \frac{1}{n+1}\,\delta^i_k\,\Gamma\indices{^p_{pj}}\,.
\end{equation*}
The Thomas symbols determine the projective structure. In dimension~2 they can be encoded in the so-called projective connection, a second-order ordinary differential equation (obtained from~\eqref{eqn:geodesic} by eliminating the external parameter)
\begin{equation}\label{eqn:projective.connection.obtained}
 y''(x)
 = -\Gamma^2_{11} +(\Gamma^1_{11}-2\Gamma^2_{12})\,y'(x)
 -(\Gamma^2_{22}-2\Gamma^1_{12})\,y'(x)^2 +\Gamma^1_{22}\,y'(x)^3,
\end{equation}
whose solutions describe geodesic curves (up to reparametrizations).
In particular, a connection $\nabla$ might come from a metric $g$ by way of the Levi-Civita connection $\nabla^g$. This is the situation that we assume in what follows.
The projective classes that we consider here can always be realized by the Levi-Civita connection of a metric~$g$.
What is even more, we assume that there are several such realizations that are essentially different (in a sense to be specified in Proposition~\ref{prop:metrization.equations} and Definition~\ref{def:metrization.space}).
\begin{definition}\label{def:metrizable.structure}
 We say that a projective structure $\projstr$ is \emph{metrizable} if there exists a metric $g$ such that $\projstr=[\nabla^g]$ where $\nabla^g$ is the Levi-Civita metric of $g$.
\end{definition}

\noindent The metric~$g$ in Definition~\ref{def:metrizable.structure} is never unique, if it exists.
Indeed, if a projective class $\projstr$ satisfies $\projstr=[\nabla^g]$ for~$g$, then any metric $\lambda g$, $\lambda\in\mathds{R}$, has the same projective structure $\projstr=[\nabla]=[\nabla^g]$.
Other, non-trivial examples might also exist, and the projective classes considered here actually admit many such realizations.

\begin{definition}
 For a metric $g$, the collection of all metrics projectively equivalent to it is called its \emph{projective class}, denoted $\projcl(g)$.
\end{definition}

\begin{remark}[Metrizability Problem]
 If two metrics belong to the same projective structure, their Thomas symbols~$\Pi\indices{^k_{ij}}$ coincide, if both metrics are expressed in the same coordinates.
 Asking whether a given projective connection~\eqref{eqn:projective.connection.obtained} represents a metrizable projective structure is referred to as the \emph{metrizability problem}.
 Let us prescribe a specific one,
  \begin{equation}\label{eqn:projective.connection.prescribed}
   y''(x) = f_0 +f_1\,y'(x) + f_2\,y'(x)^2 +f_3\,y'(x)^3\,,
  \end{equation}
  where the $f_i$ are functions. In this case the metrization problem corresponds to a system of partial differential equations on the components of the metric~$g$. It is obtained by equating the coefficients of~\eqref{eqn:projective.connection.prescribed} to~\eqref{eqn:projective.connection.obtained}.
  This system of partial differential equations is highly non-linear. However, it is well known within projective differential geometry that this non-linear system can be rewritten in linear form \cite{eastwood_2008,bryant_2008}.
\end{remark}

\noindent The metrizability problem can be turned into a system of linear partial differential equations by a suitable replacement of the unknowns. More specifically, we need to introduce weighted tensor sections.
\begin{definition}
 A \emph{(p,q)-tensor field of weight $k$} is a section in the bundle
 \[
  T^{(p,q)}M\otimes(\mathrm{vol}(M))^{\frac{k}{n+1}}
 \]
\end{definition} 
\noindent Here, $\mathrm{vol}(M)$ is the bundle of positive volume form (this presupposes that we fix an orientation, which we may, because we work locally). Also, assuming a positively oriented basis $(x^1,\dots,x^n)$, we may write $\Omega\in\mathrm{vol}(M)$ as $\Omega=f(x)\,dx^1\wedge\dots\wedge dx^n$, and can therefore think of $\Omega$ as a function, which we are going to make use of in the following. A more proper introduction to weighted tensor fields can be found in Section~2.2 of~\cite{matveev_2018}, see also~\cite{eastwood_2008,gover_2018}.
Particularly, we are going to work with (0,2)-tensor sections in $(\mathrm{vol}(M))^{\frac{k}{n+1}}\otimes T^{(0,2)}M$, where $T^{(0,2)}M=S^2T^*\!M$ denotes the symmetric (0,2)-tensors. The weight~$k$ has to be chosen suitably.
\begin{proposition}[\cite{bryant_2008,eastwood_2008}]\label{prop:metrization.equations}
 The metrizability problem, i.e.\ the condition that~\eqref{eqn:projective.connection.prescribed} is realized by~\eqref{eqn:projective.connection.obtained}, can be expressed as a system of linear partial differential equations on components of weighted tensor section $\beta$ in \smash{$(\mathrm{vol}(M))^{\frac43}\otimes S^2T^*\!M$}, which are given by
 \begin{equation}\label{eqn:a.g}
  \Psi : g\mapsto \beta\,,\qquad
  \beta_{ij} = |\det(g)|^{-\frac23}\,g_{ij}\,.
 \end{equation}
 The metrizability equations then read \cite{bryant_2008,liouville_1889}
 \begin{subequations}\label{eqn:metrizability.eqns}
 \begin{align}
    \beta_{11x} - \frac23\,f_1\,\beta_{11}+2f_0\,\beta_{12} &= 0
    \\
    \beta_{11y} + 2\beta_{12x} -\frac43\,f_2\,\beta_{11} +\frac23\,f_1\,\beta_{12}+2f_0\,\beta_{22} &= 0
    \\
    2\beta_{12y} + \beta_{22x} -2f_3\,\beta_{11} -\frac23\,f_2\,\beta_{12}+\frac43\,f_1\,\beta_{22} &= 0
    \\
    \beta_{22y} - 2f_3\,\beta_{12}+\frac23\,f_2\,\beta_{22} &= 0
 \end{align}
 \end{subequations}
\end{proposition}

\begin{remark}
 There is a second, alternative convention that turns the metrizability problem into a linear system of differential equations. Instead of $\beta$, we can use a section $\sigma$ in $(\mathrm{vol}(M))^{\frac23}\otimes S^2T^*\!M$, defined by
 \begin{equation}\label{eqn:sigma.g}
  \Phi : g \mapsto \sigma\,,\qquad
  \sigma^{ij} = |\det(g)|^{\frac13}\,g^{ij}\,.
 \end{equation}
 Both conventions can be used interchangeably (in dimension~2), since $\beta$ and $\sigma$ (in matrix representation) are simply matrix duals.
 We use the convention also adopted by~\cite{bryant_2008,matveev_2012}.
\end{remark}

\begin{definition}\label{def:metrization.space}
 The linear space of solutions to the system~\eqref{eqn:metrizability.eqns} is called the \emph{metrization space}~$\solSp$. The dimension of this space is called the \emph{degree of mobility} of the projective structure (and of any underlying metric).
\end{definition}
\noindent The metrization space~$\solSp$ contains, via~\eqref{eqn:a.g}, the metrics projectively equivalent to~$g$. However, this is not a 1-to-1 correspondence, and in fact $\Psi(\projcl)\subsetneq\solSp$, as for instance $0\in\solSp$ clearly does not correspond to a metric.
There is an interconnection between constant eigenvalues of Benenti tensors (i.e., special conformal Killing tensors) and points in $\solSp$ that do (not) correspond to metrics \cite{bolsinov_2003,manno_2018ben}.

The examples discussed in Sections~\ref{sec:examples} and~\ref{sec:main.example} have in common that they admit \emph{projective vector fields}, i.e.\ vector fields whose flow preserves geodesics up to reparametrization.
\begin{definition}
 A projective transformation is a (local) diffeomorphism of $M$ that sends (unparametrized) geodesics into (unparametrized) geodesics. An (infinitesimal) \emph{projective symmetry} is a vector field (up to multiplication by a non-zero constant) whose (local) flow acts by projective transformations.
\end{definition}
\noindent In particular, if we say that a metric admits one projective symmetry, this means that all projective vector fields are linearly dependent.
Metrics that admit one or several projective symmetries have been classified in~\cite{bryant_2008} and~\cite{matveev_2012,manno_2018}, respectively.
The simplest example of projective symmetries are homothetic vector fields, i.e.\ vector fields~$X$ that preserve a metric~$g$ up to a constant,
$\lie_Xg = \lambda g$, with $\lambda\in\mathbb{R}$.
Particularly, if we assume $\lambda=0$, this includes Killing vector fields.
A projective symmetry that is not homothetic is said to be \emph{essential} (or \emph{non-trivial}).

\section{Projectively equivalent superintegrable systems}\label{sec:equivalent.systems}
In this section, we introduce the concept of projective equivalence of second order superintegrable systems and explore some of its major properties. In particular, we will obtain an addition operation on projectively equivalent systems.
Broadly speaking, while St\"ackel transforms (i.e.\ conformal transformations of superintegrable systems) are reasonably well understood (see Section~\ref{sec:introduction}, much less is known about the projective geometry underlying superintegrability.
Superintegrable systems whose underlying geometries are projectively equivalent have, however, been the subject in some recent papers, for instance~\cite{bryant_2008,matveev_2012,manno_2018,manno_2019}. These references discuss a particular class of systems without potential. Reference~\cite{valent_2016}, on the other hand, studies (Darboux-)K{\oe}nigs systems with potential, from a global perspective. Our approach is local and includes a potential, while retaining a high degree of generality.

Let $g$ be a metric with potential $V$ and natural Hamiltonian $H=g^{ij}p_ip_j+V$.
For reasons that will become clear later on, let us introduce the following vector field\footnote{In the context of St\"ackel transforms, similarly to~\eqref{eqn:U}, it is possible to introduce a weighted scalar potential by $v=|\det(g)|^{\frac12}V$, which has similar properties as $U$.}.
\begin{definition}\label{def:U}
 The weighted vector field $U\in (\mathrm{vol}(M))^{\frac43}\otimes TM$,
 \begin{equation}\label{eqn:U}
  U[H] = |\det(g)|^{\frac23}\,\mathrm{grad}_g(V)\,,
 \end{equation} 
 is going to be referred to as the \emph{projective vector potential} of the natural Hamiltonian $H=g^{ij}p_ip_j+V$.
\end{definition}

\noindent If the Hamiltonian, from which $U$ is computed, is clear, we shall sometimes drop the mention of $H$, writing simply $U=U[H]$.

\begin{theorem}\label{thm:potential.transformation}
 Let~$g$ be a metric with potential $V$ and natural Hamiltonian $H=g^{ij}p_ip_j+V$.
 Furthermore, let~$\tilde{g}$ be a metric projectively equivalent to~$g$.
 Then
 \[
  W = \left|\frac{\det(g)}{\det(\tilde{g})}\right|^{\frac23}\,\mathrm{grad}_g(V)
 \]
 is a 0-weight vector field, $W\in TM$, and is the gradient (w.r.t.~$\tilde{g}$) of a function $\tilde{V}$,
 $W=\mathrm{grad}_{\tilde g}\tilde{V}$.
\end{theorem}
\begin{proof}
 Using the Benenti tensor $L\indices{_i^j}=\left|\frac{\det(g)}{\det(\tilde{g})}\right|^{\frac23}\tilde{g}_{ia}g^{aj}$,
 and square brackets to denote antisymmetrization,
 \[
  (dW^\flat)_{ij}
  = \tilde\nabla_{[i} \left( \tilde{g}_{j]a}W^a \right)
  = \tilde\nabla_{[i} \left( \tilde{g}_{j]a}g^{ab}W_b \right)
  = \tilde\nabla_{[i} \left( L\indices{_{j]}^b} V_b \right)\,.
 \]
 One then quickly verifies that the last expression is exactly the Bertrand-Darboux equation for the Killing tensor $K_{ij}dx^idx^j$ (w.r.t.\ the metric $g$)
 \begin{equation}\label{eqn:dW.flat}
  (dW^\flat)_{ij} = d\mathsf{K}dV = 0\,,
 \end{equation}
 where $\mathsf{K}=K_{ia}g^{aj}$ has had an index raised using the metric $g$.
 We have also used the musical isomorphism $\flat:TM\to T^*\!M$ w.r.t.\ the metric $\tilde{g}$.
\end{proof}

\noindent We come back to Equation~\eqref{eqn:dW.flat} in Section~\ref{sec:reformulate.bertrand.darboux}, where we reformulate it in terms of $\beta\in\solSp$ and the projective vector potential~$U$.

\begin{definition}\label{def:related.Hamiltonians}
 Let $g$ and $\tilde{g}$ be two projectively equivalent metrics with natural Hamiltonians $H=g^{ij}p_ip_j+V$ and $\tilde{H}=\tilde{g}^{ij}p_ip_j+\tilde{V}$.
 We say that the Hamiltonians $H_1$ and $H_2$ are \emph{projectively related} if $U[H]$ and $U[\tilde{H}]$ are equal up to a constant factor, $U[\tilde{H}]\dot=U[H]$.
\end{definition}
\noindent Here, we have introduced the notation $\dot=$ to denote equivalence up to a constant factor, $U_1\dot= U \Leftrightarrow U_1=cU_2$ with $c\ne0$.
The proof of Theorem~\ref{thm:potential.transformation} immediately yields also the following.
\begin{proposition}\label{prop:transformation.Killing.tensors}
 Let $g$ and $\tilde{g}$ be two projectively equivalent metrics with natural Hamiltonians $H=g^{ij}p_ip_j+V$ and $\tilde{H}=\tilde{g}^{ij}p_ip_j+\tilde{V}$.
 If $H$ admits the quadratic integral $F=K^{ij}p_ip_j+W$, then $\tilde{H}$ admits the Killing tensor
 \[
  \tilde{K} = \left(\frac{\det(\tilde{g})}{\det(g)}\right)^{\frac23}\,K
 \]
 The corresponding potential remains unchanged, i.e.\ the Hamiltonian $\tilde{H}$ admits the integral of motion
 \begin{equation*}
  \tilde{F} = \tilde{K}^{ij}p_ip_j+W\,.
 \end{equation*}
\end{proposition}

The proof is given in Section~\ref{sec:proof.integrals.transformation}.

\begin{remark}
 For free Hamiltonians (without potential), Proposition~\ref{prop:transformation.Killing.tensors} is given in~\cite{painleve_1894}, see also~\cite{topalov_2003} and references therein.
 Proposition~\ref{prop:transformation.Killing.tensors} reflects the projective equivalence of the Killing equation, see~\cite{bryant_2008} and references therein. In~\cite{bryant_2008}, the corollary is stated for free Hamiltonians and the isomorphism $K\mapsto\tilde{K}$ is referred to as the canonical isomorphism.
\end{remark}

\noindent Definitions~\ref{def:U} and~\ref{def:related.Hamiltonians} together with Theorem~\ref{thm:potential.transformation} and Proposition~\ref{prop:transformation.Killing.tensors} establish a projective equivalence of superintegrable systems.
Indeed, let $H=g^{ij}p_ip_j+V$ be a superintegrable Hamiltonian. Then, for a metric $\tilde{g}$ projectively equivalent to $g$, we can construct a potential $\tilde{V}$ corresponding to $V$ by requiring $U[\tilde{H}]=U[H]$ for the transformed Hamiltonian $\tilde{H}=\tilde{g}^{ij}p_ip_j+\tilde{V}$.
Therefore, we arrive at the following.
\begin{definition}\label{def:equivalent.systems}
 Let $H=g^{ij}p_ip_j+V$ and $\tilde{H}=\tilde{g}^{ij}p_ip_j+\tilde{V}$ be two superintegrable Hamiltonians that are projectively related.
 If $U[\tilde{H}]=U[H]$, then we say that $H$ and $\tilde H$ give rise to \emph{projectively equivalent superintegrable systems}.
 For two projectively equivalent metrics $g,g'$ we say that their superintegrable systems are \emph{projectively equivalent} if they have the same projective potential $U$.
\end{definition}

\begin{remark}
It is easily verified that this is indeed an equivalence relation. In fact, reflexivity, symmetry and transitivity are straightforwardly confirmed.\smallskip

\noindent In Definition~\ref{def:equivalent.systems}, we require equality in $U[\tilde{H}]=U[H]$. In Definition~\ref{def:related.Hamiltonians}, on the other hand, we have only equality up to a constant factor, $U[\tilde{H}]\dot=U[H]$. While the second is a priori a weaker requirement, it turns out that both are very similar criteria. The reason is that if the metric $g$ admits the potential $V$, the Hamiltonian already admits the family $cV$ of potentials. Indeed, if the Bertrand-Darboux Equation~\eqref{eqn:bertrand.darboux} is satisfied for a potential $V$ (and a family of Killing tensors $K$), it is also satisfied for any constant multiple of $V$ (and the same family of Killing tensors). However, the scalar part of the integrals of motion will transform accordingly (see Example~\ref{ex:trivial.case} below).
In what follows, adopting a common convention from the theory of superintegrability, we are often going to speak of the potential $V$, when indeed we have an entire family of such potentials in mind. In this case, the equivalence in Definition~\ref{def:equivalent.systems} becomes, effectively, equivalence up to a constant factor (up to renaming the parameters). Therefore, in such situations, Definitions~\ref{def:related.Hamiltonians} and~\ref{def:equivalent.systems} differ only in the requirement of superintegrability.
\end{remark}

\begin{example}[trivial transformations]\label{ex:trivial.case}
 Let $H=g^{ij}p_ip_j+V$ be a natural Hamiltonian arising from a metric~$g$. Then for any $\lambda\ne0$ ($\lambda\in\mathbb{R}$) the metric $g'=\lambda g$ is projectively equivalent to~$g$ and gives rise to a family of natural Hamiltonians $H'_{\lambda,\mu}=G'+V'_\mu$ with $G'=g'_{ij}\xi^i\xi^j$ and a family of potentials $V'_\mu=\mu V$.
 For any choice of $\lambda,\mu$ the Hamiltonian~$H_{\lambda,\mu}$ is equivalent to~$H$.
 This is easily verified as
 \[
  U[H'] = \lambda^{\frac43}\mu\,U[H]\,.
 \]
 If the Hamiltonian $H$ admits the integrals $J^{(\alpha)}=K^{(\alpha)}_{ij}\xi^i\xi^j+W^{(\alpha)}$ ($\alpha=1,\dots,2n-1$), then the transformed Hamiltonian~$H'$ admits the integrals
 \smash{$I^{(\alpha)}=K^{(\alpha)}_{ij}\xi^i\xi^j+\mu W^{(\alpha)}$}.
\end{example}

\noindent For brevity, we sometimes use the following abbreviation: We say that two potentials $V^{(1)}$ and $V^{(2)}$ are projectively equivalent, if the corresponding Hamiltonians are, for which we assume that the underlying metrics are clear.

\begin{remark}
 Note that the Hamiltonians admitted by two projectively equivalent metrics are, in general, not projectively equivalent.
 For instance, the flat generic system and the (isotropic) harmonic oscillator are not projectively equivalent.
 The metric, for both systems, is $g=dx^2+dy^2$, and the (non-degnerate) potentials are given by
 \begin{align*}
  V_\text{gen} &=\omega^2(x^2+y^2)+\frac{a}{x^2}+\frac{b}{y^2}+c
  && \text{The generic system} \\
  V_\text{osc} &=\omega^2(x^2+y^2)+a x+b y+c
  && \text{The isotropic harmonic oscillator.}
 \end{align*}
 The claim is easily checked, first verifying that both these potentials are compatible with $\partial_x^2$ and $\partial_y^2$, but that the third compatible Killing tensor is, respectively,
 \[
  K_\text{gen} = (y\,dx -x\,dy)^2
  \qquad\text{or}\qquad
  K_\text{osc} = dxdy\,.
 \]
\end{remark}

\begin{remark}
 Note that the equivalence $U[H_1]=U[H_2]$ in the definition is to be understood as an equivalence of the respective families of admissible potentials, such that
 \[
  V_\text{osc'} = \omega^2(x^2+y^2)+a'(x+y) +b'(x-y) +c
 \]
 and $V_\text{osc}$ give rise to equivalent (and actually coinciding) superintegrable systems.
\end{remark}

\noindent Having introduced the notion of projective equivalence of second order superintegrable systems, let us now turn our attention towards a different, but related, problem.
While so far, we have been concerned with how to transform one given superintegrable system into another, we are now going to assume that we are already provided with a pair\footnote{For simplicity, we restrict to a pair of two metrics from which we construct the family of superintegrable systems. Analogously, one might define the addition for more than two given superintegrable Hamiltonians. In Section~\ref{sec:main.example} we indeed use Hamiltonians that are defined from a basis of the metrizability space.} of projectively equivalent, (second-order) superintegrable Hamiltonians $H_1,H_2$.
Let us denote the underlying metrics by $g_1$ resp.\ $g_2$ (we assume they are non-proportional projectively equivalent metrics), and the potentials by $V_1,V_2$.
Then, any metrics of the form
\begin{equation}\label{eqn:g.t}
  g_t = \frac{
          \frac{g_1}{\det(g_1)^{2/3}}+t\,\frac{g_2}{\det(g_2)^{2/3}}
       }{
          \det\left(
            \frac{g_1}{\det(g_1)^{2/3}}+t\,\frac{g_2}{\det(g_2)^{2/3}}
          \right)^2
       }\,.
\end{equation}
is projectively equivalent to $g_1$ and $g_2$.
Formula~\eqref{eqn:g.t} is highly non-linear, and a priori we should expect the same for the potential.
However, due to the linearizability of the metrizability equations, we actually obtain a linear formula (the proof is given in Section~\ref{sec:proof.addition.potential}).
\begin{theorem}\label{thm:addition}
 Let $g_1,g_2$ be projectively equivalent, linearly independent metrics that give rise to projectively related superintegrable natural Hamiltonians $H_1,H_2$ (with potentials $V^{(1)},V^{(2)}$.
 Then the family~\eqref{eqn:g.t} of metrics gives rise to a family of superintegrable systems with potentials
 \begin{equation*}
  V_t = V^{(1)}+tV^{(2)}\,.
 \end{equation*}
\end{theorem}

\noindent Next, by the same token, let $g_1,\dots,g_m$ be projectively equivalent, linearly independent metrics that define projectively equivalent superintegrable systems with potentials $V^{(1)},\dots,V^{(m)}$. Then they define a family of superintegrable systems on the metrics analogous to~\eqref{eqn:g.t},
\begin{equation}\label{eqn:g.t.arbitrary}
  g[t_1,t_2,\dots,t_m] =
       \frac{
          \sum_it_i\frac{g_i}{\det(g_i)^{2/3}}
       }{
          \det\left(
            \sum_it_i\frac{g_i}{\det(g_i)^{2/3}}
          \right)^2
       }\,.
\end{equation}
The resulting metric admits the potential
$V[t_1,t_2,\dots,t_m] = \sum_i t_i V^{(i)}$.

Giving the linearity of~\eqref{eqn:metrizability.eqns}, we can thus define an ``addition'' of superintegrable systems as follows.
Let us denote by $S_i=(g_i,V^{(i)})$ the system with metric $g_i$ and potential $V^{(i)}$, for $i\in\{1,2,\dots,r\}$ with $r\in\mathds{N}$.
Let $\beta_i=\Psi(g_i)$ for each $i$.
Then we define, for constants $t_i\in\mathds{R}$,
\begin{equation}\label{eqn:addition.systems}
  \sum_i t_i S_i := \left(\Psi^{-1}\left(\sum_i t_i \beta_i\right),\sum_i t_i V^{(i)}\right)\,.
\end{equation}

\begin{remark}
In the literature, addition theorems for superintegrable systems are studied, e.g., in~\cite{tsiganov_2008,tsiganov_2009}. However, these addition theorems do not seem to have any apparent link to the additive property discussed here.
\end{remark}

\subsection{Proof of Proposition~\ref{prop:transformation.Killing.tensors}}\label{sec:proof.integrals.transformation}
Proving Proposition~\ref{prop:transformation.Killing.tensors}, we also obtain an intrinsic motivation for Definition~\ref{def:U}.
For free Hamiltonians (without potential) the Proposition already appears in~\cite{painleve_1894}, see~\cite{topalov_2003} for a more modern formulation. Using this classical result, consider a metric $g$ which admits a projectively equivalent metric $\hat{g}$. The following integral of motion is admitted by~$\hat{g}$ (indices of~$\hat{g}$ are raised using $g$)
\begin{equation}\label{eqn:integral.g.hat}
 F_{\hat{g}} = \underbrace{\det(g)^{\frac23}\,\frac{\hat{g}^{ij}}{\det(\hat{g})^{\frac23}}}_{=K^{ij}_{(\alpha)}}\,p_ip_j + W\,.
\end{equation}
The integral has to satisfy~\eqref{eqn:poisson}.
Therefore, the quadratic part of the integral of motion will be still given by solutions of the metrization equations, i.e.\ elements of $\solSp$. Thus it remains to study the latter equation in~\eqref{eqn:poisson.split},
\begin{equation*}
\{K^{ij}_{(\alpha)}p_ip_j,V\}+\{G,W^{(\alpha)}\} = 0\,.
\end{equation*}
Using the weighted tensor fields $\beta^{(\alpha)}$, corresponding to $F^{(\alpha)}$ via~\eqref{eqn:a.g} and~\eqref{eqn:integral.g.hat}, we have
\begin{equation}\label{eqn:transformation.Killing}
 K^{ij}_{(\alpha)}p_ip_j=\det(g)^{\frac23}\,\beta^{(\alpha)}_{kl}g^{ki}g^{lj}p_ip_j
\end{equation}
and we thus obtain from~\eqref{eqn:poisson.split}
\begin{equation}\label{eqn:W.k}
 W^{(\alpha)}_k = (\det g)^{\frac23}\,g^{im}g^{jn}\beta^{(\alpha)}_{mn}\,V_{,i}\,g_{jk}
 = (\det g)^{\frac23}\,\beta^{(\alpha)}_{mk}\,V^m\,,
\end{equation}
where subscripts after a comma denote derivatives, e.g.\ $V_{,i}$ is the $i$-th component of the differential $dV$ of $V$.
An inspection of Equation~\eqref{eqn:W.k} motivates Definition~\ref{def:U}, in view of Theorem~\ref{thm:potential.transformation}.
Proposition~\ref{prop:transformation.Killing.tensors} now immediately follows from~\eqref{eqn:transformation.Killing}.



\subsection{A reformulation of the Bertrand-Darboux Equation}\label{sec:reformulate.bertrand.darboux}

If the metric is clear or irrelevant, we again denote (covariant) derivatives by comma, such that superscripts denote components of the gradient of a function, and subscript components of its differential.
%
Let $g$ be a metric on the manifold $M$ admitting the potential $V:M\to\mathds{R}$ and thus the Hamiltonian $H=g^{ij}p_ip_j+V$.
With the definition of $U$, we have the formula
\begin{equation}\label{eqn:WqU}
   W^{(\alpha)}_k = \beta^{(\alpha)}_{mk}\,U^m\,,
\end{equation}
where $U^m$ does, by definition, not depend on $\alpha$.
The Bertrand-Darboux Equation~\eqref{eqn:bertrand.darboux} thus becomes
\begin{equation}\label{eqn:invariant.BD}
  U^i\beta^{(\alpha)}_{i[j,k]}-U\indices{^i_{\,,[j}}\beta^{(\alpha)}_{k]i} = 0\qquad\forall\alpha\,.
\end{equation}

\subsection{Proof of Theorem~\ref{thm:addition}}\label{sec:proof.addition.potential}
Let us take the projective vector potential defined in~\eqref{eqn:U}. We may lower one index, but this will depend on our choice of the metric~$g$, among all metrics projectively equivalent to $g$.
However, take an integral of motion satisfying~\eqref{eqn:poisson}. We may write the differential of the scalar part $V^{(\alpha)}$ as
\begin{equation}\label{eqn:dV.qU}
 \nabla_iV^{(\alpha)} = \beta^{(\alpha)}_{ij}\,U^j\,,
\end{equation}
which is indeed independent of the choice of the metric~$g$.
In turn, we may replace the metric~$g$ by its corresponding solution of~\eqref{eqn:metrizability.eqns}, which we may denote $\beta=\Psi(g)=\sum k_\alpha \beta^{(\alpha)}$.
One straightforwardly realizes that the integrability relation for~\eqref{eqn:dV.qU} is~\eqref{eqn:invariant.BD} and therefore satisfied.

\begin{remark}
 Equation~\eqref{eqn:dV.qU} has some similarity with the following observation:
 If we are provided with a pair of St\"ackel equivalent Hamiltonians, $H_1=g_1^{ij}p_ip_j+V_1$ and $H_2=g_2^{ij}p_ip_j+V_2$, we can form the product $V_ig_i$ for $i=1,2$.
 However, due to the St\"ackel equivalence, $H_2=\phi H_1$ for some function $\phi$, and thus
 $
  V_2g_2 = (\phi V_1)\frac{g_1}{\phi} = V_1 g_1\,.
 $
 For further details see \cite{kalnins_2011,capel_2014,kalnins_2018} for instance.
\end{remark}

\section{Examples}\label{sec:examples}
We have already mentioned a few simple examples of projectively equivalent systems earlier, and will now turn our attention to more interesting ones. Our main application, however, is going to be the classification, up to St\"ackel equivalence, of superintegrable systems with one, essential projective symmetry, see the following section.
In the examples here, we are going to look at superintegrable systems on Darboux-K{\oe}nigs
\footnote{Excluding constant curvature spaces, there exist four Darboux-K{\oe}nigs systems. They are called K{\oe}nigs metrics in~\cite{valent_2016}, and Darboux metrics in~\cite{kalnins_2002,kalnins_2003}, referring to a note by G.~K{\oe}nigs \cite{koenigs} in the multi-volume tome of G.~Darboux \cite{darboux_1897}, respectively.}
metrics, which have already been studied in several papers~\cite{kalnins_2002,kalnins_2003,bryant_2008,valent_2016}.
Also, we consider systems on conformally equivalent geometries, and in particular systems that are both St\"ackel and projectively equivalent.

\subsection{Darboux-K{\oe}nigs systems}

The Darboux-K{\oe}nigs metrics are projectively equivalent~\cite{bryant_2008}.
They admit the following (degenerate) superintegrable potentials
\begin{align*}
    g_1 &= \frac{a\cos(x)+b}{\sin^2(x)}\,(dx^2\pm dy^2)
    &V_1 &= \frac{c_1}{a\cos(x)+b}+c_2 \\
    g_2 &= (ae^{-x}+be^{-2x})\,(dx^2\pm dy^2)
    &V_2 &= \frac{c_1}{ae^x+b}+c_2 \\
    g_3 &= \left( \frac{a}{x^2}+1 \right)\,(dx^2\pm dy^2)
    &V_3 &= \frac{c_1}{x^2+a}+c_2 \\
    g_4 &= x\,(dx^2\pm dy^2)
    &V_4 &= \frac{c_1}{x}+c_2
\end{align*}
Taking into account these potentials, the full natural Hamiltonians are projectively related also in the sense of Definition~\ref{def:equivalent.systems}.
This is easily seen by using the representation of Darboux-K{\oe}nigs metrics in the form given in~\cite{bryant_2008}, for which we find the Hamiltonian with potential to be
\[
     H = \frac12\,e^{3x}\,dx^2-De^x\,dy^2 + c_1e^x+c_2\,.
\]
The corresponding vector potential is
\[
     U = 2^{-\frac13}\,e^{\frac43\,x}\,\frac{c_1}{|D|^{\frac23}}\,\partial_x\,,
\]
which, for any value of $D\ne0$ is the same up to rescaling.

\subsection{Constant curvature metrics}
Consider the flat metric $g=dx^2+dy^2$. There exist 20 different superintegrable systems for this metric \cite{kalnins_2001}. The corresponding (families of) potentials are compatible with different subspaces of the space of Killing tensors and thus the Hamiltonians connected with different potentials cannot be projectively equivalent in the sense of Definition~\ref{def:equivalent.systems}.
However, systems on different constant curvature spaces can be equivalent. For instance, take the flat metric with the so-called generic potential,
\[
  H = (p_x^2+p_y^2) + \omega^2\,(x^2+y^2)+\frac{a}{x^2}+\frac{b}{y^2}+c\,.
\]
It is compatible with a 3-parameteric Killing tensor
\[
  K = C_1\,\left( y^2\,dx^2-2xy\,dxdy+x^2\,dy^2 \right) + C_2\,dx^2 + C_3\,dy^2\,.
\]
Likewise, the Hamiltonian given by the metric
\[
  g' = \frac{1}{(x^2+y^2+2)^2}\,\left((y^2+2)dx^2-2xy\,dxdy+(x^2+2)\,dy^2\right)\,,
\]
which has sectional curvature~1, and the potential
\[
  V' = \omega^2\,(x^2+y^2)+(y^2+1)\,\frac{a}{x^2}+(x^2+1)\,\frac{b}{y^2}+c
\]
is compatible with the family of Killing tensors
\[
  K' = \frac{K}{(x^2+y^2+1)^2}\,.
\]
Both have the same vector potential
\[
  U = -2\,\left( (a-\omega^2x^4)\,\frac{dx^2}{x^3} + (b-\omega^2y^4)\,\frac{dy^2}{y^3} \right)\,.
\]

\subsection{St\"ackel and projectively equivalent systems}
Let us consider a pair of projectively equivalent metrics with
\[
  g_2 = \phi\,g_1
\]
for a function~$\phi$. The corresponding potentials (Hamiltonians) shall be denoted by $V_2,V_2$ ($H_2,H_1$), respectively.
Then,
\[
  \gamma = \frac{\det(g_2)}{\det(g_1)} = \phi^2\,.
\]
This means that the potentials are related by
\[
  \nabla^{g_2}V_2 = \det(g_2)^{-\frac23} U = \gamma^{-\frac23} \nabla^{g_1}V_1 = \phi^{-\frac43} \nabla^{g_1}V_1\,,
\]
and, rewritten in terms of differentials, we have
\begin{equation}\label{eqn:dV2.projective}
  dV_2 = \phi^{-\frac13}\,dV_1\,.
\end{equation}
On the other hand, from the St\"ackel equivalence of the systems, we get $V_2=\frac{V_1}{\phi}$. Therefore, we have also
\begin{equation}\label{eqn:dV2.conformal}
  dV_2 = \frac{\phi\,dV_1 - V_1\,d\phi}{\phi^2}\,.
\end{equation}
Combining~\eqref{eqn:dV2.projective} and~\eqref{eqn:dV2.conformal}, we find the requirement
\[
  \phi(1-\phi^{\frac23})\,dV_1 = V_1\,d\phi\,,
\]
which is necessary for the Hamiltonians $H_1,H_2$ being simultaneously projective and St\"ackel equivalent.
Solving for the potential $V_1$, we obtain
\begin{equation}\label{eqn:potential.formula}
  V_1 = \frac{c\,\phi}{(1-\phi^{\frac23})^{\frac32}}\qquad c\in\mathbb{R}\,.
\end{equation}
As an interesting side remark, we observe
\begin{equation}\label{eqn:wedge.formula}
  d\phi\wedge dV_1 = d\phi\wedge dV_2 = 0\,.
\end{equation}
which alternatively we could have concluded from
\[
  0 = ddV_2
  = -\frac13\,\phi^{-\frac43}\,d\phi\wedge dV_1 + \phi^{-\frac13}\,ddV_1
  = -\frac13\,\phi^{-\frac43}\,d\phi\wedge dV_1\,.
\]

\noindent We will find (many) examples of Hamiltonians that are both St\"ackel and projectively equivalent, when we discuss our main example in Section~\ref{sec:main.example}.
However, it should be stressed that Equations~\eqref{eqn:potential.formula} and~\eqref{eqn:wedge.formula} are only necessary conditions, and not sufficient for superintegrability. This is illustrated by the following concrete example:
Let us consider the pair of projectively equivalent metrics
\begin{align*}
	 g_1 &= dx^2+dy^2 \\
	 g_2 &= \frac{dx^2+dy^2}{(1+x^2+y^2)^2}\,,
\end{align*}
meaning $\phi=(1+x^2+y^2)^{-2}$.
\begin{proposition}
 There are no potentials $V_1,V_2$ such that the Hamiltonians $H_1=g_1^{ij}p_ip_j+V_1$ and $H_2=g_2^{ij}p_ip_j+V_2$ are superintegrable and simultaneously projectively and St\"ackel equivalent.
\end{proposition}
\noindent Let us now prove the claim. To this end, because of~\eqref{eqn:wedge.formula}, we have
\[
 y\frac{\partial V}{\partial x}-x\frac{\partial V}{\partial y} = 0
 \qquad\Rightarrow\qquad
 V=V(\rho)\,,\quad \rho=x^2+y^2\,.
\]
However, using~\eqref{eqn:potential.formula}, we can also explicitly integrate for the potential, i.e.\ for~\eqref{eqn:potential.formula}, which is compatible with the Killing tensor
\[
 K = (y\,dx-x\,dy)^2
\]
in addition to the metric, but this is not sufficient for superintegrability.

\section{The St\"ackel classes admitted by metrics with one, essential projective symmetry}\label{sec:main.example}
As a main application of the discussion of Section~\ref{sec:equivalent.systems}, we now establish the St\"ackel classes of metrics with one, essential projective symmetry. This extends the description, and classification (up to isometries, without potential) of such systems in~\cite{matveev_2012,manno_2018,manno_2019}.

\noindent The core questions are: \emph{What potentials are admitted by metrics that admit one, essential projective symmetry in dimension~2? Which of these are equivalent under St\"ackel transforms?}

\noindent From the literature, we can adopt the following description of metrics with one, essential projective symmetry, the space of solutions to~\eqref{eqn:metrizability.eqns} is 3-dimensional and its basis is given, via~\eqref{eqn:sigma.g} from the following three (projectively equivalent) metrics
\begin{subequations}\label{eqn:supint.generators}
\begin{align}
 g_1 &= (x+y^2)\,dxdy \\
 g_2 &= -2\frac{x+y^2}{y^3}\,dxdy+\frac{(x+y^2)^2}{y^4}\,dy^2 \\
 g_3 &= \frac{y^2+x}{(3x-y^2)^6} \left(
 9\,(y^2+x)\,dx^2
 -4y\,(9x+y^2)\,dxdy
 +12x\,(y^2+x)\,dy^2
 \right)\,.
\end{align}
\end{subequations}

\noindent To state the result, we need the following fact, which we find in the existing literature.
\begin{theorem}[\cite{matveev_2012,manno_2018,manno_2019}]\label{thm:metrics.one.essential.vf}
 The 2-dimensional (pseudo-)Riemannian metrics that admit exactly one, essential projective symmetry are projectively equivalent.
 They are parametrized, up to isometries, by points on the 2-sphere (with 6 points removed),
 Metrics that admit a second-order superintegrable system and exactly one, essential projective symmetry are, locally around almost every point, isometric to a metric $g=\Psi^{-1}(\beta)$ where
 \begin{equation}\label{eqn:matveev.system}
  \beta = \cos(\theta)\sin(\varphi)\,\beta_1
          +\cos(\theta)\cos(\varphi)\,\beta_2
          +\sin(\theta)\,\beta_3
 \end{equation}
 with $\theta\in(-\frac{\pi}{2},\frac{\pi}{2})$, $\varphi\in(0,2\pi]$, but $\varphi\not\in\{0,\frac{\pi}{2},\pi,\frac{3\pi}{2}\}$ if $\theta=0$.
 The $\beta_i$ are obtained, via Equation~\eqref{eqn:sigma.g}, from~\eqref{eqn:supint.generators}.
\end{theorem}

\begin{corollary}[\cite{manno_2019}, using results from~\cite{matveev_2012,manno_2018}]
 Starting from~\eqref{eqn:matveev.system} and alternatively to $(\beta_1,\beta_2,\beta_3)$, the following basis of~$\solSp$ can be constructed:
 \begin{align*}
 \beta &= \cos(\theta)\sin(\varphi)\,\beta_1
      +\cos(\theta)\cos(\varphi)\,\beta_2
      +\sin(\theta)\,\beta_3\,,
 \\
 \bar{\beta} &=\sin(\theta)\sin(\varphi)\,\beta_1
	   +\sin(\theta)\cos(\varphi)\,\beta_2
	   -\cos(\theta)\,\beta_3\,,
 \\
 \hat{\beta} &=-\cos(\varphi)\,\beta_1+\sin(\varphi)\,\beta_2\,.
 \end{align*}
 The triple $(\beta,\bar{\beta},\hat{\beta})$ gives rise, via~\eqref{eqn:sigma.g} and~\eqref{eqn:integral.g.hat}, to a metrizable superintegrable system for the free Hamiltonian $G=g^{ij}p_ip_j$.
\end{corollary}

\begin{remark}
By inspection of the references~\cite{manno_2018,manno_2019} indeed any point in $\mathds{R}^3\setminus\{0\}$ corresponds to a second-order maximally superintegrable metric. Using the flow of the (unique) projective symmetry, the parametrization of Proposition~\ref{thm:metrics.one.essential.vf} is then obtained via identification of isometric metrics.
The axes in $\mathds{R}^3$ can be chosen such that they represent the metrics for which the unique projective symmetry is actually homothetic, leading to the restrictions on $\theta,\varphi$ in Proposition~\ref{thm:metrics.one.essential.vf}.
\end{remark}


\noindent In addition to the action of the isometry group (which is already accounted for in \cite{manno_2018,manno_2019}), the St\"ackel transform acts on the classification space.
We determine the orbits under this equivalence operation using the method outlined in Section~\ref{sec:introduction}.
In order to do so, we need to determine the potentials for a basis of~$\solSp$. We choose the basis $(\beta_1,\beta_2,\beta_3)$.
\begin{lemma}\label{la:generator.potentials}
The metrics~\eqref{eqn:supint.generators} admit the Hamiltonian $H^{(a)}=g_a^{ij}p_ip_j+V^{(a)}$ given, respectively, by
\begin{align*}
 S_1:\quad g_1 &= (x+y^2)dxdy
 & V^{(1)} &= \frac{c_1}{x+y^2}+\frac{c_2\,y}{x+y^2}+c_3\,\frac{y(y^2-3x)}{x+y^2}+c_4
 \\
 S_2:\quad g_2 &= -2\,\frac{x+y^2}{y^3}\,dxdy+\frac{(x+y^2)^2}{y^4}\,dy^2
 & V^{(2)} &= \frac{y}{x+y^2}\,a_1 +\frac{y^2}{x+y^2}\,a_2 -\frac{y^2(y^2-3x)}{x+y^2}\,a_3+a_4
 \\
 S_3:\quad g_3 &= \frac{y^2+x}{(3x-y^2)^6}
	\left(\star\right)
 & V^{(3)} &= \frac{y(3x-y^2)}{x+y^2}\,b_1 +\frac{(3x-y^2)^2}{x+y^2}\,b_2 +\frac{(3x-y^2)^3}{x+y^2}\,b_3 +b_4
\end{align*}
\[
 \star = 9\,(y^2+x)\,dx^2-4y\,(9x+y^2)\,dxdy+12x\,(y^2+x)\,dy^2
\]
These are non-degenerate second order maximally superintegrable systems, and therefore the scalar parts $V^{(\alpha)}$ represent the maximal possible families of potentials.
\end{lemma}
\begin{proof}
 The claim is verified by a straightforward computation. Indeed, the expressions have been obtained by an explicit integration of the Bertrand-Darboux equation for the metrics~\eqref{eqn:supint.generators}.
\end{proof}

\noindent Consider now a generic metric~$g$, given by~\eqref{eqn:matveev.system} via~\eqref{eqn:a.g}. Solving~\eqref{eqn:bertrand.darboux} for~$g$ explicitly is, although conceptually straightforward, hard to do explicitly. Indeed, in view of the rather complicated formula~\eqref{eqn:g.t.arbitrary}, the equations turn out to be cumbersome and lengthy. However, using the techniques from Section~\ref{sec:equivalent.systems}, we are able to complete this seemingly hard problem almost trivially.
\begin{theorem}\label{thm:main.application}
 The superintegrable systems whose underlying metric admits one, essential projective symmetry are non-degenerate second-order superintegrable systems. They are parametrized by the 2-sphere except 6 exceptional points where the projective symmetry becomes homothetic.
 Explicitly, the metric~$g=\Psi^{-1}(\beta)$ from Proposition~\ref{thm:metrics.one.essential.vf}, specified by
 \begin{subequations}\label{eqn:general.superintegrable}
 \begin{equation}\label{eqn:2.sphere.beta}
  \beta = \cos(\theta)\sin(\varphi)\,\beta_1 + \cos(\theta)\cos(\varphi)\,\beta_2 + \sin(\theta)\,\beta_3\,,
 \end{equation}
 admits the potential
 \begin{equation}\label{eqn:2.sphere.potential}
  V = \cos(\theta)\sin(\varphi)\,V^{(1)} + \cos(\theta)\cos(\varphi)\,V^{(2)} + \sin(\theta)\,V^{(3)}\,.
 \end{equation}
 \end{subequations}
 The six exceptional points are $\beta\in\{\pm\beta_1,\pm\beta_2,\pm\beta_3\}$, i.e.\ if $\cos(\theta)=0$, or if $\sin(\theta)=0$ and $\sin(\varphi)\cos(\varphi)=0$.
 Up to (diffeomorphism and) St\"ackel transformations, there exist only three such superintegrable systems with essential projective symmetry. They have St\"ackel type (111,11), (21,2) or (21,0). Specifically, using the angles as in~\eqref{eqn:general.superintegrable}, the St\"ackel class is
  \begin{itemize}[align=left]
   \item[(111,11)] generically, with exception of the points satisfying
   \begin{equation}\label{eqn:stackel.nongenericity}
    \sin(\theta)
    \left(\tan(\theta) - \frac{2^{2/3}}{108}\,\frac{\sin^3(\varphi)}{\sin^2(\varphi)}\right)
    = 0
   \end{equation}
   \item[(21,2)] if $\tan(\theta) = \frac{2^{2/3}}{108}\,\frac{\sin^3(\varphi)}{\sin^2(\varphi)}$
   \item[(21,0)] if $\sin(\theta) = 0$
  \end{itemize}
  Note that these three cases exclude the points where the projective symmetry becomes homothetic. These points correspond to superintegrable systems of type (21,2) if $\beta=\pm\beta_2$ and (3,11) otherwise.
\end{theorem}

\noindent Figure~\ref{fig:orbits} illustrates Theorem~\ref{thm:main.application}. The gray (equator) and black curves show the orbits where the St\"ackel type is not generic. The black curve depicts points where the St\"ackel type is (21,2), except where it intersects with the equator. This intersection point has type (3,11), as have the north and south poles. The other points on the equator have St\"ackel type (21,0).

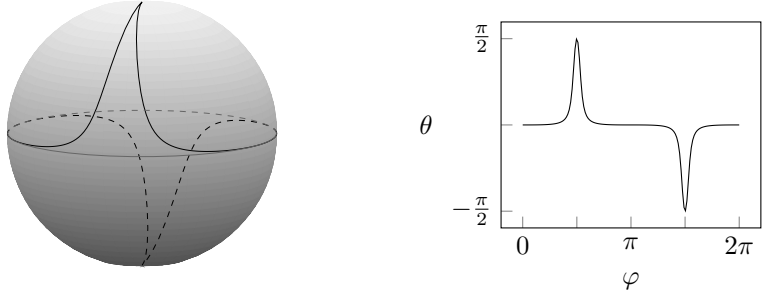
\begin{figure}
\includegraphics[scale=0.75]{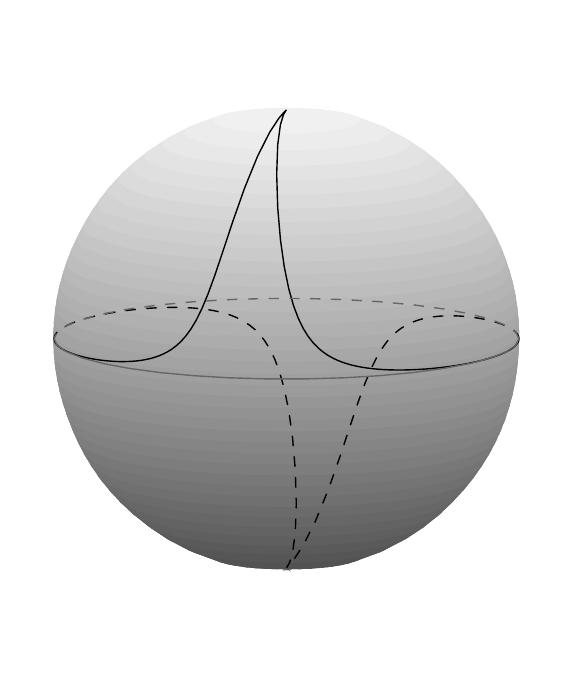}
\hspace{1cm}
\raisebox{0.4cm}{
\begin{tikzpicture}
    \begin{axis}[width=5cm, 
		 xtick={0, 1.571, 3.142, 4.712, 6.283},
		 xticklabels={$0$, $ $, $\pi$, $ $, $2\pi$},
		 ytick={-1.571,0,1.571},
		 yticklabels={$-\frac{\pi}{2}$,$ $,$\frac{\pi}{2}$},
		 tick pos=left,
		 xlabel={$\varphi$},
		 ylabel={$\theta$}, ylabel style={rotate=-90}]
        \def\azimuth{(x)}
        \def\elevation{(rad(atan(2^(2/3)/108*sin(deg(\azimuth))^3/cos(deg(\azimuth))^2)))}
        \addplot[domain=0:2*pi, samples=150] {\elevation};
    \end{axis}
\end{tikzpicture}
}
 \caption{The ``St\"ackel type degeneration'' within the classification space of metrics with one, essential projective symmetry. The gray circle is the orbit of St\"ackel type (21,0).
 The darker orbit is the subvariety of St\"ackel type (21,2), and degenerates into type (3,11) at the north and the south pole as well as at the intersections with the gray circle (=equator). The graph on the right visualises the function $\theta(\varphi)$ obtained from the second bracket in~\eqref{eqn:stackel.nongenericity}.}\label{fig:orbits}
\end{figure}

The theorem is going to be proven in the following section. Afterwards, in Section~\ref{sec:contractions}, we are going to comment on the interrelations of the St\"ackel classes appearing in Theorem~\ref{thm:main.application}.

\begin{remark}
Before we turn to the proof of Theorem~\ref{thm:main.application}, let us remark that in some sense the theorem covers all interesting cases of metrics with one, projective symmetry.
In fact, we are going to see:\smallskip

\emph{Metrics with a (non-trivial) homothetic symmetry are of constant curvature or are multiples of $g_1,g_2$ or $g_3$.}
\smallskip

\noindent Here is the proof:\smallskip

\noindent We assume degree of mobility at least~2, since otherwise we are in a trivial situation. Having exactly one (up to rescaling), homothetic projective vector field~$v$, we can follow the strategy in~\cite{matveev_2012}, i.e.\ we use that the Lie derivative of a metric~$g$ w.r.t.\ its homothetic vector field $v$ satisfies
\[
 \lie_vg = \lambda\,g\,,\quad\text{i.e.,}\ \ 
 \lie_v\beta = \mu\,\beta\,,\quad\forall\,\beta\in\solSp\,.
\]
We can solve this system of differential equations in a way analogous to the procedure in \cite{matveev_2012}, using the Dini-Bolsinov-Matveev-Pucacco theorem on normal forms of pairs of projectively equivalent metrics, see \cite{bolsinov_2009} which is an Appendix to \cite{matveev_2012}.
Therefore, see~\cite{matveev_2012}, either the metric is a multiple of $g_1,g_2$ or $g_3$, or there exists a 2-dimensional $\lie_v$-invariant subspace $\solSp_0\subset\solSp$ of the metrization space $\solSp$ such that $\lie_v|\solSp_0=\lambda$.
For Liouville metrics, the following Frobenius system is found, analogously to~\cite{matveev_2012},
\begin{align*}
X' &= 0
& Y' &= 0
&
(v_1)_x &= \frac12
& (v_1)_y &= 0
&
(v_2)_x &= 0
& (v_2)_y &= \frac12\,.
\end{align*}
After possibly a translation in $x$, $y$, and a rescaling of the coordinates, we obtain
\[
g_a = dx^2+dy^2\,,\quad
v = x\partial_x+y\partial_y
\]
Secondly, for Complex Liouville metrics, $v_1+i\,v_2 = -\frac32\,z+\text{constant}$ and
$h_z = 0$,
which yield (after obvious transformations)
\[
g_b = dz^2-d\bar{z}^2\,,\quad
v = z\partial_z+\bar{z}\partial_{\bar{z}}
\]
Lastly, in case of Jordan block normal forms for the metrics, the Frobenius system is equivalent to
\[
(v_1)_x = \frac32\,,\quad
v_1 = \frac32\,(x+Y)\,,\quad
0 = (v_1)_y = \frac32\,Y'
\]
and thus, after a translation in $x$, the metric $g$ is
\[
g_c = x\,dxdy\,,\quad\text{or}\quad
g_d = -2\,\frac{x\,dxdy}{y^3}+\frac{x^2}{y^4}\,dy^2\,.
\]
The metrics $g_a,g_b,g_c,g_d$ have constant curvature.
\end{remark}

\subsection{Proof of Theorem~\ref{thm:main.application}}\label{sec:homothetic.case}
The first part, i.e.\ Equations~\eqref{eqn:general.superintegrable}, are straightforward using our previously developed methods.
A priori, we have to integrate the Bertrand-Darboux equation~\eqref{eqn:bertrand.darboux} for any other metric of the projective class that we consider. This would, indeed, be a quite demanding task for a generic metric.
We can circumvent that issue using Theorems~\ref{thm:potential.transformation} and~\ref{thm:addition}. Indeed, instead of the explicit integration, we can exploit the generating systems determined in Theorem~\ref{thm:main.application}. This is enough to reconstruct completely and straightforwardly the admissible potential of any other metric of the projective class.
Specifically, this is achieved using Equations~\eqref{eqn:W.k} and~\eqref{eqn:WqU} together with that in Theorem~\ref{thm:potential.transformation}.
In \cite{manno_2018} it has been proven that the metrics with an essential projective symmetry admit freely superintegrable systems (i.e., without potential).
The admissible potentials for the metrics $g_1,g_2,g_3$ from~\eqref{eqn:supint.generators} have been found in Lemma~\ref{la:generator.potentials}.

\begin{lemma}[The generating systems]\label{la:generating.systems}
The metrics corresponding to~\eqref{eqn:matveev.system} give rise to projectively equivalent non-degenerate superintegrable systems with the projectively superintegrable system specified by the data $(\projcl,\solSp,U)$ where the projective potential is
\begin{equation}\label{eqn:U.c}
U = -\frac{c_3 (y^4 +3x^2) + c_2 (y^2 -x) + 2c_1 y}{(y^2 + x)^{\frac53}}\,\,\partial_x\,\,
-\frac{2c_3	y^3 + c_2 y + c_1}{(y^2 + x)^{\frac53}}\,\,\partial_y\,.
\end{equation}
The explicit potentials $V^{(1)},V^{(2)},V^{(3)}$ are (respectively for the generator metrics $g_1,g_2,g_3$ which provide a basis of~$\solSp$)
\begin{align*}
V^{(1)} &= -\frac{(y^2+3x)y\,c_3}{y^2+x}+\frac{y\,c_2}{y^2+x}+\frac{c_1}{y^2+x}+c_4 \\
V^{(2)} &= -\frac{2^{\frac23}}{4}\,
\left(\frac{(y^2-3x)y^2\,c_3}{y^2+x}+\frac{2y^2\,c_2}{y^2+x}+\frac{2y\,c_1}{y^2+x}+c_4\right) \\
V^{(3)} &= \frac{2^{\frac13}}{8}\,
\left(\frac{(y^2-3x)^3\,c_3}{y^2+x}+\frac{2(y^2-3x)^2\,c_2}{y^2+x}
+\frac{8(y^2-3x)y\,c_1}{y^2+x}-8c_1+c_4\right)\,,
\end{align*}
\end{lemma}
\begin{proof}
It is a priori not clear whether two projectively equivalent metrics admit the same projective potential~$U$.
However, knowing the generating systems above, we can exploit our knowledge about the transformation behavior from Theorem~\ref{thm:potential.transformation}. It allows us to deduce the corresponding parameters in the three potentials by comparing the parameters and their respective functional coefficients.
The free constants are, of course, not unique, but we can choose the set $(c_i)$, for instance, for which we find the explicit expression~\eqref{eqn:U.c}.
%
\end{proof}

\noindent It only remains to prove the functional independence of the integrals following from this construction. But already the parts of the integrals quadratic in the momenta are functionally independent (this is easily checked using the Jacobian, cf.~\cite{manno_2019}). Therefore, also the full integrals are functionally independent.
\medskip

\noindent So only the St\"ackel classes for the superintegrable systems corresponding to each point of the 2-sphere need to be computed. We continue as follows: First, compute the St\"ackel type for the six exceptional points where the projective symmetry becomes homothetic. Then, we continue by first considering subsets of the classification space defined by using two of the generating systems. In a final step, the generic case is going to be considered.
Each of these intermediate considerations will, for sake of clarity, be presented as an independent lemma.
In order to keep notation short, we will mostly only refer to the point $\beta$ on the classifying 2-sphere, i.e.\ w.r.t.\ the representation~\eqref{eqn:2.sphere.beta}. Without explicitly saying it, we will then work with the natural Hamiltonian $H=g^{ij}_\beta p_ip_j+V_\beta$ where $g_\beta$ is the metric corresponding to~$\beta$ via~\eqref{eqn:a.g}, and where $V_\beta$ is the respective potential, corresponding with $\beta$, according to~\eqref{eqn:2.sphere.potential}.
We begin with the six exceptional cases (note that the St\"ackel type is unchanged if we reverse the sign, $\beta\mapsto-\beta$ and $V\mapsto-V$).
\begin{lemma}\label{prop:staeckel.classes.homothetic}
 The systems for $\beta=\beta_1$ and $\beta=\beta_3$ are of St\"ackel type (3,11). The system for $\beta=\beta_2$ is of St\"ackel type (21,0).
\end{lemma}
\begin{proof}
 The statement follows by a straightforward computation using computer algebra such as Maple\texttrademark\ or Sagemath\nocite{maple,sagemath}.
\end{proof}

\noindent Starting from the generating systems and Theorem~\ref{thm:main.application}, we can compute the quadratic algebra associated to the respective system, following the directions outlined in Section~\ref{sec:introduction}.
For practical purposes, it is most convenient to use~\eqref{eqn:general.superintegrable} together with the generators (instead of the more complicated general ones).
Moreover, we can use the addition operation defined in~\eqref{eqn:addition.systems}.
\begin{lemma}
 Let $\beta_{ij}(t_1,t_2)=t_1\beta_i+t_2\beta_j$ and consider the corresponding added systems, using the addition of systems as in~\eqref{eqn:addition.systems}.
 Then the systems for $\beta_{12}=\sin(\psi)\beta_1+\cos(\psi)\beta_2$ are of St\"ackel type (21,0) except when $\cos(\psi)=0$. 
 The systems $\beta_{i3}=\sin(\psi)\beta_i+\cos(\psi)\beta_3$, for $i\in\{1,2\}$, are of type (111,11) except when $\sin(\psi)\cos(\psi)=0$, i.e.\ $\psi\neq k\frac{\pi}{2}$ where $k\in\mathds{Z}$. The exceptions are the generator cases discussed in Proposition~\ref{prop:staeckel.classes.homothetic}.
\end{lemma}
\begin{proof}
 The results for the systems $\beta_{12}$ are obtained straightforwardly following \cite{kress_2007}.
 For the remaining two families, we notice that the leading part cubic in the integrals of motion $I^{(1)}$ and $I^{(2)}$ has Discriminant $\Delta=0$ if and only if $\sin(\psi)\cos(\psi)=0$.
\end{proof}

\noindent Now, let us turn to the most generic case.
\begin{lemma}\label{la:essential.systems}
 Generically, the systems (again, addition is defined as in~\eqref{eqn:addition.systems}) for
 \[
  \beta_{123}=\cos(\theta)\cos(\varphi)\,\beta_1
	  +\cos(\theta)\sin(\varphi)\,\beta_2+\sin(\theta)\,\beta_3
 \]
 are of type (111,11).
 Degeneration occurs if and only if $\sin(\theta)\cos(\theta)=0$ or
 \begin{equation}\label{eqn:new.degenerate.case}
   \tan(\theta)=\frac{2^{2/3}}{108}\,\frac{\sin(\varphi)^3}{\cos(\varphi)^2}\,.
 \end{equation}
 In the latter case, the type is $(21,2)$ except when both $\sin(\theta)=0$ and $\sin(\varphi)=0$.
\end{lemma}
\begin{proof}
 The first part of the proof is straightforward with computer algebra, because we can use the cubic discriminant to identify cases where the type is of the form $(3,\ast)$ or $(21,*)$ with $*$ indicating that we only consider the leading part of the cubic at this step. If the leading part admits three distinct roots, we are done since this implies already the type (111,11).
 Next, in order to prove that cases~\eqref{eqn:new.degenerate.case} are of type (21,0), we use a new representation.
 Letting $\beta = t_1\,\beta_1 + t_2\,\beta_2 + t_3\,\beta_3$,
 the condition for degenerate cases translates into
 \begin{equation}\label{eqn:condition.rewritten}
  (108\,t_1^2\,t_3-2^{\frac23}\,t_2^3)\,t_3 = 0\,.
 \end{equation}
 If $t_3=0$ we end up in the $\langle \beta_1,\beta_2\rangle$-plane, so let $t_3\ne0$. We may then choose a representative with $t_3=1$ and $t_1=\pm1$ as this does not change the St\"ackel type.
 From~\eqref{eqn:condition.rewritten} we infer
 $t_2=\sqrt[3]{2^{\frac13}\,54}$,
 and then it is straightforward to show that the resulting system is of type (21,2).
\end{proof}

\subsection{Contractions}\label{sec:contractions}
We conclude this section by a review of the interrelations of the St\"ackel classes we have found to be realized in Theorem~\ref{thm:main.application}. This interrelation is provided by what is known as (B\^ocher or \.In\"on\"u-Wigner) contractions \cite{kalnins-I,kalnins_2013,capel_2015}.
These are singular limits of families of coordinate transformations on a Lie group or its Lie algebra. Under certain conditions, the limit of such singular transformations yields a new superintegrable system (the transformed system is not necessarily defined on the same space as the initial one), see~\cite{inonu_1953,saletan_1961}.
The explicit contractions that relate systems belonging to one of the St\"ackel classes appearing in Theorem~\ref{thm:main.application} can be found in~\cite{kalnins_2014}.

\begin{example}
Consider the special orthogonal group $SO(3)$. In coordinates, the structure relations on the Lie algebra are given by the antisymmetric Levi-Civita symbol, i.e.\ $[\ell_i,\ell_j] = \epsilon^{ijk}\ell_k$ where $\ell_i$ are the angular momentum operators. Now transform
\[
 \ell_1\to\hat\ell_1=\varepsilon\ell_1\,,\quad
 \ell_2\to\hat\ell_2=\varepsilon\ell_2\,,\quad
 \ell_3\to\hat\ell_3=\ell_3\,.
\]
This becomes singular for $\varepsilon=0$. But it still has a well-defined limit on the level of structure constants yielding in the limit $\varepsilon\to0$ the 2D Euclidean group $E(2)$.
While defined on the level of the Lie algebra, a realization on the level of coordinates is given in~\cite{kalnins_2014} as follows: Let $SO(3)$ act on the 2-sphere $\mathbb{S}^2\subset\mathbb{R}^3$ with $\ell_1=x_2p_3-x_3p_2$ etc.\ and the Hamiltonian given by $H=\sum_i\ell_i^2$. Restriction to the sphere means that we have $\sum_ix_i^2=1$ and $\sum_ix_ip_i=0$. The contraction from $SO(3)$ to $E(2)$ can then be implemented on the coordinates as
\[
 y_1\to\frac{x_1}{\varepsilon}\,,\quad
 y_2\to\frac{x_2}{\varepsilon}\,,\quad
 y_3\to1\,.
\]
\end{example}

\noindent In the theory of superintegrability, contractions are interesting in at least two respects. First, for non-degenerate second-order 2-dimensional superintegrable systems, there is a generic system on $\mathbb{S}^2$ (referred to as [S9]), from which all other superintegrable systems on $\mathbb{S}^2$ and $\mathbb{E}^2$ can be obtained by subsequent singular limits, via contractions \cite{kalnins_2013} (an analogous result exists in dimension~3 \cite{capel_2015}; the phenomenon was first observed by B{\^o}cher \cite{bocher_1894}).
%
Second, taking contractions as directed transformations of superintegrable systems (or rather, their St\"ackel classes), a hierarchy of St\"ackel classes can be written down \cite{kalnins-I,kalnins_2013,capel_2015}, with (111,11) being the uppermost node of the resulting graph, as established in~\cite{kalnins_2013}.
This graph can, in fact, be related to contractions of hypergeometric orthogonal polynomials in the Askey scheme~\cite{kalnins_2013}, revealing a link between special functions and superintegrable systems.

Figure~\ref{fig:contractions} shows the graph of contractions between non-degenerate 2-dimensional St\"ackel equivalence classes. It shows that the systems appearing in Theorem~\ref{thm:main.application} form a subgraph, which relate to the curves, and their intersections, as shown in Figure~\ref{fig:orbits}. Specifically, the black and gray curves in Figure~\ref{fig:orbits} correspond with the nodes (21,2) and (21,0) in Figure~\ref{fig:contractions}, respectively; except for their intersection points and the north and south poles, which correspond to the node (3,11).
Contractions can thus be interpreted in terms of subvariety inclusions, if we adopt a description as in~\eqref{eqn:condition.rewritten} within the proof of Lemma~\ref{la:essential.systems}.

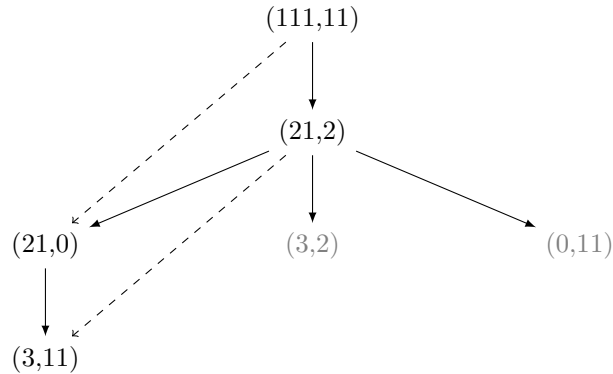
\begin{figure}
\begin{center}
 \begin{tikzpicture}[sibling distance=10em,edge from parent/.style={draw,-latex},
  every node/.style = {shape=rectangle, rounded corners, align=center,top color=white,bottom color=white}]]
  \node (111-11) {(111,11)}
    child { node (21-2) {(21,2)}
      child { node (21-0) {(21,0)}
        child { node (3-11) {(3,11)} }
      }
      child { node[color=gray] {(3,2)} }
      child { node[color=gray] {(0,11)} }
    };
    \draw[dashed,->] (111-11) -- (21-0);
    \draw[dashed,->] (21-2) -- (3-11);
 \end{tikzpicture}
 \medskip
 \caption{The hierarchy of non-degenerate superintegrable systems in dimension~2, see~\cite{kalnins_2013}. Gray indicates St\"ackel classes not appearing in Theorem~\ref{thm:main.application}. Solid arrows are taken from~\cite{kalnins_2013}, while dashed arrows indicate further degenerations along curves in Figure~\ref{fig:orbits}.}\label{fig:contractions}
\end{center}
\end{figure}

\section{Discussion and Conclusion}\label{sec:conclusions}
In this paper we have studied (second-order) superintegrable systems from the viewpoint of projective differential geometry, which provided us with a concept of equivalence for superintegrable systems on 2-dimensional (pseudo-)Riemannian manifolds that share the same geodesics up to reparametrization.
As we have seen this concept is in many respects similar to the that of St\"ackel equivalence, which may be viewed as its conformal counterpart.

A number of examples have been presented to illustrate how the techniques offered by this approach can be exploited, such as for the construction of superintegrable systems and for verifying their projective equivalence. Particularly, we have found a formula for the potentials of simultaneouly St\"ackel and projectively equivalent systems, see~\eqref{eqn:potential.formula}.
Finally, as a concrete application, we have classified (up to St\"ackel equivalence) all superintegrable systems whose underlying metric admits one, non-trivial (i.e., essential) projective symmetry. Figure~\ref{fig:orbits} provides a concise geometric interpretation for this classification in terms of subvarieties on the ambient space $\mathbb{R}^3\supset\mathbb{S}^2$.

Directions for further research include, for instance, a generalization to higher dimensions, and a combination with efforts to classify superintegrable systems in terms of algebraic varieties. In dimension 2, for flat geometries, such a classification already exists \cite{kress_2019}, and a generalisation to higher dimension is currently underway. In particular, the algebraic-geometric classification space should be expected to be endowed not only with the action of the isometry group, but also the action of projective and conformal transformations.

\section*{Acknowledgement}
The author would like to thank Rod Gover and Jonathan Kress for enlightening discussions on superintegrability, St\"ackel transform and  their relation with conformal and (metric) projective geometry, as well as Vladimir Matveev and Gianni Manno for sharing their knowledge on projective differential geometry and metrizability. Special thanks go to Gianni Manno for proofreading the manuscript and for several helpful suggestions.
Moreover, the author is grateful towards Konrad Sch\"obel, Joshua Capel and Galliano Valent for valuable comments and remarks.
%
The author is a postdoc research fellow of Deutsche Forschungsgemeinschaft (DFG) and acknowledges travel support from DFG, the University of Auckland and Istituto Nazionale di Alta Matematica (INdAM). The author thanks FSU Jena, the University of Pavia and the University of Auckland for hospitality.\medskip

\noindent Funded by Deutsche Forschungsgemeinschaft (DFG, German Research Foundation) Project number 353063958

\nocite{*}

\sloppy
\printbibliography

\end{document}